\documentclass[journal,onecolumn]{IEEEtran}
\usepackage{graphicx} 
\usepackage{amsmath}
\usepackage{amssymb}
\usepackage{amsthm}

\newtheorem{mydef}{Definition}
\newtheorem{thrm}{Theorem}
\newtheorem{remark}{Remark} 
\newtheorem{lem}{Lemma}
\usepackage{multicol}
\usepackage{multirow}
\usepackage[english]{babel}   
\usepackage{color}
\usepackage{multicol}
\usepackage[T1]{fontenc}
\usepackage{algorithm,algorithmic}
\usepackage{caption}
\usepackage{subcaption}
\newcommand\norm[1]{\left\lVert#1\right\rVert}
\begin{document}
\title{Model Based Robust Control Law for \\Linear Event-triggered System}
\author{Niladri Sekhar Tripathy,~\IEEEmembership{}
        I. N. Kar~\IEEEmembership{}
        and~Kolin Paul~\IEEEmembership{}
\thanks{Niladri Sekhar Tripathy and I. N. Kar are with the Dept. of  Electrical Engineering, Indian Institute of Technology Delhi, New Delhi, India e-mail: (niladri.tripathy@ee.iitd.ac.in,  ink@ee.iitd.ac.in).}
\thanks{Kolin Paul is with the Dept. of Computer science \& Engineering, Indian Institute of Technology Delhi, New Delhi, India e-mail: kolin@cse.iitd.ac.in. }
\thanks{Manuscript received Month xx, 20xx; revised Month xx, 20xx.}}

%
%

\markboth{}%
{Shell \MakeLowercase{\textit{et al.}}: Bare Demo of IEEEtran.cls for Journals}
\maketitle

\begin{abstract}
This paper proposes a framework to design an event-triggered based robust control law for linear uncertain system. The robust control law is realized through both  static and dynamic event-triggering approach to reduce the computation and communication usages.   Proposed control strategies ensure stability in the presence of bounded matched and unmatched system uncertainties.  Derivation of event-triggering rule with a non-zero positive inter-event time and corresponding stability
criteria for uncertain event-triggered system are the key contributions of this paper. The efficacy of proposed algorithm is carried out through a comparative study of simulation results. 


\end{abstract}

\section{Introduction}

Aperiodic sensing, communication and computation play a crucial role for controlling resource constrained cyber-physical systems.  It is shown in \cite{so:16}-\cite{newbib} that aperiodic sampling has more benefits over periodic sampling, which motivates control researchers towards event-triggered control.  In event-triggered control,  sensing, communication and computation happens only when any predefined event condition is violated.  Event-triggered control strategy finds applications in different control problems like tracking  \cite{so:15}, estimation  \cite{so:14}-\cite{new4} etc. Event-triggered system is modeled as a perturbed system in continuous  and  discrete time domain respectively \cite{something6,something4}. Also the behaviour of such system is described by an impulsive dynamics in literature \cite{hi1}-\cite{hi}. To achieve larger average inter-event time, \cite{something} proposes a dynamic event-generating rule over the previous approach \cite{something6}, which makes event-triggered strategy more computationally efficient and predictable. The input to state stability (ISS) property \cite{new1}-\cite{so:three} is exploited to prove the closed loop stability and to define triggering condition for event-triggered system. Sahoo  et al \cite{hi,sao} proposed an event based adaptive control approach for uncertain systems. They use a neural network to estimate the nonlinear function to generate the control law. In event based robust control problems,  the uncertainty is mainly considered in the communication channel in the form of time-delay or data-packet loss \cite{so:12}. The main shortcoming of the classical event-triggered system lies in the fact that one must know the exact model of the plant apriori. A plant with an uncertain (system) model is a more realistic scenario and has far greater significance. However, there are open problems of designing a control law and triggering conditions to deal with system uncertainties. These uncertainties mainly arise due to system parameter variations, unmodeled dynamics, disturbances etc. which require the design of robust controller. An optimal control approach to robust controller design for the uncertain system has been reported in \cite{so:one}-\cite{so:2}. The applications include tracking problem in robot manipulator \cite{something5}-\cite{newci}, set-point regulation in CSTR system etc.  To achieve an optimal solution to the robust control problem there is a need to minimize a cost functional. In this direction, a non-quadratic cost functional is utilized to solve robust control problem with input constraint \cite{so:1}-\cite{so:2}. In the above mentioned approach, event-trigger based implementation of robust control law is not considered which is essential in the context of networked control systems (NCS).
\\ This paper considers the robust control strategies of linear uncertain system  with limited state and input information. The limited state information is considered to address the channel unreliability or bandwidth constraint which is a very common phenomena in NCS. 
To capture the channel unreliability and bandwidth limitations, event-triggered control strategy is adopted \cite{vg}. With limited information, existing robust control results  in \cite{so:one}-\cite{so:onetwo} can not be simply extended to the event-triggered system is the primary motivation for this work. This paper proposes a novel event based robust control strategy for both matched and unmatched uncertain systems. In a matched system, it is assumed that the unknown uncertainty is in the span of control input matrix. This assumption does not hold in case of the unmatched system. Here control input is computed and updated only when an event is generated.
\begin{figure}
\begin{center}
\includegraphics[width=7.5 cm,height=8.5 cm]{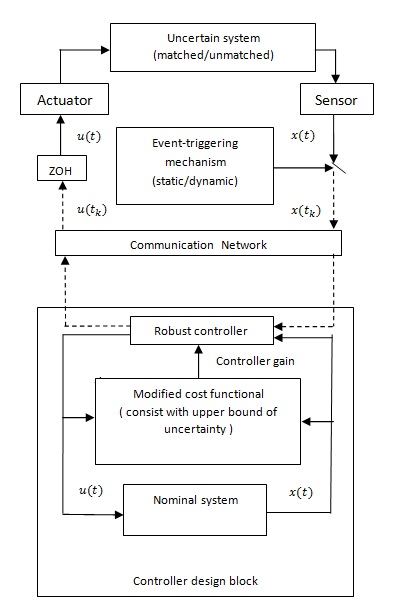}
\caption{Conceptual Block diagram of proposed  event-trigger based robust control. Dotted line represents the aperiodic information transmission through the communication channel. Here  $x(t_{k}), u(t_{k})$ are representing eventual state and contol input respectively.}\label{fi:11}
\end{center}
\end{figure} 
 A conceptual block diagram of the proposed framework is illustrated in Figure \ref{fi:11} where the system, sensor and actuator are co-located but the controller is connected thorough a communication network. A dedicated computing unit monitors the event condition at the sensor end. The aperiodic state transmission to controller and control input update instant $\{t_{k}\}_{k\in I}$ over the  network is decided by the same event-triggering law. For simplification, it is assumed that there is no communication, computation and actuation delay in the system.  To design robust control law, an equivalent optimal control problem is formulated with an appropriate cost functional which takes care of the upper bound of system uncertainty.  The nominal system dynamics is used to compute the optimal controller gain which minimizes the cost-functional. A zero-order-hold (ZOH) at the actuation end holds the last transmitted control input until the transmission of next input. The analysis of this  system is done in continuous time domain. The proposed method is verified for both static and dynamic event-triggering rule. In both cases corresponding triggering rule and their stability criteria for matched and unmatched uncertain system have been derived. The advantage of the proposed control strategy is that it significantly reduces the number of control input transmission and computation in spite of system uncertainties.
\subsubsection*{Summary of contribution} The main contributions of this paper are summarized as follows.
\begin{itemize}
\item Defining an optimal control problem to design a robust control law for both matched and unmatched uncertain system.
\item Deriving static and dynamic event-triggering rule for uncertain system using the upper bound of system uncertainty.
  \item    Ensuring stability of closed loop system using ISS Lyapunov function.
  \item  Deriving a positive non-zero lower bound of inter-execution time.
  \item A comparative study is carried out to verify the efficacy of event-triggered robust control law.
\end{itemize}

\subsubsection*{Organization of paper} The paper is organized as follows. In Section \ref{se12}, we briefly review of ISS, matched and  unmatched system uncertainty and optimal approach to robust control design. Section \ref{se123} and \ref{se14} discuss the optimal control approach to solve the robust stabilization problem for event-triggered system with system uncertainty.  Here both static and dynamic event triggering conditions are stated in the form of theorems and their corresponding proofs are reported. In section \ref{se123} we also give a mathematical expression of the minimum positive inter-event time.  Two different examples with simulation results are discussed in Section \ref{se15} to validate the proposed control algorithm.  Section \ref{se16} concludes the paper.
\subsubsection*{Notation}  The notation $\|.\|$ is used to denote the Euclidean norm of a vector $x\in{\mathbb{R}^{n}}$. Here $\mathbb{R}^{n}$ denotes the $n$ dimensional Euclidean real space and $\mathbb{R}^{n\times {m}}$ is a set of all $(n\times{m})$  real matrices. $\mathbb{R}^{+}_{0}$ and $\mathbb{I}$ denote the all possible set of positive real numbers and non-negative integers. $X\leq{0}$, $X^{T}$ and $X^{-1}$ represent the negative definiteness, transpose and inverse of matrix $X$, respectively. Symbol $I$ represents an identity matrix with appropriate dimensions and time $t_{\infty}$ implies $+\infty$. Symbols $\lambda_{min}(P)$ and $\lambda_{max}(P)$ denote  the minimum and maximum eigenvalue of  symmetric matrix $P\in \mathbb{R}^{n\times n} $ respectively. A function $f$: $\mathbb{R}_{\geq 0}\rightarrow \mathbb{R}_{\geq 0}$ is $K_{\infty}$ if it is continuous and strictly increasing and it satisfies $f(0)=0$ and $f(s)\rightarrow\infty$ as $s\rightarrow\infty$.
\section{Preliminaries \& problem statement}\label{se12}
\subsection{Preliminaries}
\subsubsection{Input to state stability} 
In state space form, a linear system with disturbance $d(t)\in\mathbb{R}^{n}$  is expressed as 
\begin{equation}\label{one}
\dot{x}(t)=Ax(t)+Bu(t)+d(t)
\end{equation}
where  $x(t)\in{R^{n}}$, $\ u(t)\in{R^{m}}$  are used to represent system's states and control input respectively. For simplification from now onwards, $x(t)$ and $u(t)$ are denoted by $x$ and $u$ respectively. Assuming disturbance function $d(t)$ as an external input and it is always bounded by a known function $d_{m}(t)$ i.e. $\|d(t)\|\leqslant d_{m}(t)$. The above system (\ref{one}) is said to be ISS with respect to $d(t)$ if there exist an ISS Lyapunov function. For analyzing ISS of the above system, following definition is introduced \cite{new1}-\cite{so:three}.     
\begin{mydef}\label{de1}
A continuous function $ V: \mathbb{R}^{n}\rightarrow \mathbb{R}$  is an ISS Lyapunov function for system (\ref{one}) if there exist class $\emph{k}_{\infty}$ functions $ \alpha_{1}, \alpha_{2}, \alpha_{3}$ and  $ \gamma$ for all $x, d \in\mathbb{R}^{n}$ and it satisfy
\begin{eqnarray}\label{deft1}
\alpha_{1}(\|x(t)\|)\leq V(x(t))\leq \alpha_{2}(\|x(t)\|)\\
\nabla{V}(x)\dot{x}\leq -\alpha_{3}{(\|x(t)\|)}+\gamma{(\|d(t)\|)}\label{deft2}
\end{eqnarray}
\end{mydef}
\subsubsection{System uncertainty} In (\ref{one}), system matrix $A$  and input matrix $B$  may depend on some uncertain parameters. In general system uncertainty is classified in two  categories namely matched and unmatched uncertainty. They are defined as follows: \\
\emph{System with matched uncertainty}:
A linear system having system-uncertainty is described by 
\begin{equation}\label{eq:five}
\dot{x}=A(p)x+Bu
\end{equation}
where $p\in{P} $ is an uncertain parameter vector. The system (\ref{eq:five}) has matched uncertainty if  there exists a bounded uncertain matrix $\phi({p})\in\mathbb{R}{^{m\times{n}}}$ such that 
\begin{equation}\label{mis66}
A(p)-A(p_{0})=B\phi({p})
\end{equation} 
for any $p\in{P}$, where $p_{0}$ is known nominal parameters and $A(p_{0})$ is nominal system matrix. In other words system uncertainty is assumed to be in the range space of input matrix $B$.  The condition (\ref{mis66}) is made to simplify the derivation of stability results. It is assumed that there exits a positive semi definite matrix $F$ to represent the upper bound of the uncertainty i.e., 
\begin{equation}\label{un190}
\phi(p)^{T}\phi(p)\leq{F}
\end{equation}
 for all $ p\in{P} $.\\ 
\emph{System with unmatched uncertainty}: System (\ref{eq:five}) have unmatched uncertainty if its uncertainty  is not in the range of input matrix, $B$.  In general system uncertainty ($A(p)-A(p_{0})$) can be decomposed in matched and unmatched component using pseudo-inverse $B^{+}$ of input matrix $B$ \cite{matrixthe}. Using $B^{+}=(B^{T}B)^{-1}B^{T}$, the 
uncertainty introduced in (\ref{eq:five}) can be written as
\begin{eqnarray}\label{mis1001}
A(p)-A(p_{0})&=&BB^{+}(A(p)-A(p_{0}))+(I-BB^{+})(A(p)-A(p_{0}))
\end{eqnarray}
Here $BB^{+}(A(p)-A(p_{0}))$ is matched and 
$(I-BB^{+})(A(p)-A(p_{0}))$ is an unmatched component of system uncertainty. It is assumed that $\forall \ p\in P$ their exist $F_{u}>0$ and $H>0$, such that following holds:
\begin{eqnarray}\label{mis10}
(A(p)-A(p_{0}))^{T}(B^{+})^{T}B^{+}(A(p)-A(p_{0}))\leq F_{u}\\
\alpha^{-2}((A(p)-A(p_{0}))^{T}(A(p)-A(p_{0}))\leq H\label{miswe}
\end{eqnarray}
Here the scalar $\alpha\geq0$ is a design parameter. 
Now to stabilize (\ref{eq:five}) with matched uncertainty (\ref{mis66}) (or unmatched uncertainty (\ref{mis1001})), we need to design a robust controller. The primary aim of robust controller is stated as follows:
\subsubsection{Robust control problem}
Find a state feedback control law $u=Kx$ such that the uncertain system (\ref{eq:five}) is stable with (\ref{mis66}) or (\ref{mis1001}) for any $p\in{P}$.
\\To solve the above mentioned  robust control problem, this paper has adopted an optimal control approach. The essential idea is to  compute the optimal control input for the nominal system which minimizes the modified cost functional. The cost functional is called modified cost functional as it depends on the upper bound of system uncertainty. The obtained optimal control input for nominal system is shown to be a robust control input for the actual uncertain system.  Here the system (\ref{eq:five}) may have  matched or unmatched uncertainty. In both cases their corresponding nominal dynamics and cost functional are considered as follows: 
\begin{itemize}
\item Nominal dynamics and cost functional for matched uncertain system are described as
\begin{equation}
\dot{x}=A(p_{0})x+Bu_{1}\label{mawe}
\end{equation}
\begin{equation}\label{eq:45}
J_{m}=\int_{0}^{\infty}(x^{T}F_{m}x+x^{T}Qx+u_{1}^{T}Ru_{1})dt
\end{equation}
 The matrix $F_{m}\geq 0$ is the upper bound of matched uncertainty and it is defined as
\begin{equation}
\phi(p)^{T}R\phi(p)\leq F_{m}
\end{equation}
\item Auxiliary dynamics and cost functional for the unmatched uncertain system are defined as
\begin{equation}\label{mismatched1}
\dot{x}=A(p_{0})x+Bu_{2}+\alpha(I-BB^{+})v
\end{equation}
\begin{equation}\label{mismatched2}
J_{u}=\int_{0}^{\infty}(x^{T}(F_{u}+\rho^{2}H+\beta^{2}I)x+u_{2}^{T}u_{2}+\rho^{2}v^{T}v)dt
\end{equation}
\end{itemize}
The  state feedback control input $u_{1}=K_{1}x$   is used to stabilize (\ref{mawe}). Similarly control inputs $u_{2}=K_{2}x$ and $ v=Lx$ are used for (\ref{mismatched1}). The control input $v$ is an auxiliary control input which ensures robustness in-spite of unmatched uncertainty. Now to design a robust control law using optimal control approach, following lemma is introduced  \cite{so:one}-\cite{so:onetwo}, \cite{so:2}. 
\begin{lem}\label{lemm1}
 Suppose we have an optimal control solution of nominal system (\ref{mawe}) for matched system [(\ref{mismatched1}) for unmatched]  with a modified cost functional (\ref{eq:45}) [(\ref{mismatched2}) for unmatched]. Then the optimal control law for the nominal system  will be the  robust control solution of the original system (\ref{eq:five})  for all bounded system uncertainty (\ref{mis66})     [unmatched uncertainty (\ref{mis1001})].
\end{lem}
\begin{proof}
A detailed explanation is given in Appendix \ref{ap121}.
\end{proof} 
\subsection{Problem description and statement}
In this paper, we realize the above mentioned robust control problem through an aperiodic state feedback control law. This formulation helps to realize such controller in the network control domain with  limited state information. The aperiodic control input computation and actuation instant is determined through a predefined state-dependent event condition. This event condition is derived from a stability criteria. Now if $\{t_{k}\}$ represents (aperiodic state transmission, control input computation and actuation instant) the event occurring instant, then the event-based state feedback control input will be 
\begin{equation}
u(t_{k})=Kx(t_{k}),\label{in12}
\end{equation}
which replaces the general continuous time state feedback control law $u(t)=Kx(t)$. 
To solve the robust control problem through a aperiodic control law (\ref{in12}) the uncertain linear system  (\ref{eq:five})  can be rewritten as
  \begin{eqnarray}\label{nst1}
\dot{x}=A(p)x(t)+Bu(t_{k})
\end{eqnarray}
for any $p\in P$. Adopting the concepts introduced in \cite{something6}, the event-based  closed loop system (\ref{nst1}) reduces to
\begin{eqnarray}
\dot{x}=A(p)x+BK(x+e)\label{eq:four}
\end{eqnarray}
Here the variable $e\in{R^{n}}$ is referred to as measurement error and is defined as
\begin{equation}\label{err1}
e(t)=x(t_{k})-x(t),  \forall{t\in{[t_{k},t_{k+1})},k\in{\mathbb{I}}}
\end{equation}
Using (\ref{eq:four}) and (\ref{mis66})  the event-triggered system with matched uncertainty is described as 

 \begin{equation}
\dot{x}=A(p_{0})x+Bu_{1}+B(\phi(p)x+K_{1}e)\label{eq:eight}
\end{equation}
 For unmatched uncertainty (\ref{mis1001}),   the event-triggered system (\ref{eq:four}) is written as
  \begin{eqnarray}\label{mis12}
\dot{x}&=&A(p_{0})x+BK_{2}(x+e)+\alpha(I-BB^{+})L(x+e)+BB^{+}(A(p)-A(p_{0}))x\nonumber
\\&&+(I-BB^{+})(A(p)-A(p_{0}))x-\alpha(I-BB^{+})L(x+e)
\end{eqnarray} 
\textbf{Problem statement:} The design of the controller gain $K_{1}$ for (\ref{eq:eight}) and $K_{2}$ along with $L$ for (\ref{mis12}) to stabilize  an uncertain event-triggered system (\ref{eq:eight}) or (\ref{mis12}) such that the entire closed loop system is ISS with respect to its measurement error $e$ (defined in (\ref{err1})) is the problem that we attempt to solve.  
The solution to the problem is derived in two steps. Firstly, we design a controller using Lemma \ref{lemm1} and then define an event-triggering rule such that the  closed loop system (\ref{eq:eight}) [or (\ref{mis12})] is ISS. These two solution steps are briefly discussed in next two subsections.    
\subsubsection{Controller design} 
The system (\ref{mawe}) is the nominal dynamics of (\ref{eq:eight}) for matched system. Now using Lemma \ref{lemm1} the optimal controller gain $K_{1}$ of (\ref{mawe}) which minimizes the cost functional (\ref{eq:45}) will be the robust solution of (\ref{eq:eight}). Similarly for   unmatched system optimal gain $ K_{2} ,\ L$ of (\ref{mismatched1}), that minimizes cost functional (\ref{mismatched2}), is the robust solution for (\ref{mis12}).
\begin{description}
\item[Step 1] For matched system, control input $u(t)$ is designed by minimizing $J_{m}$. Suppose $V(x)$ be a Lyapunov function for  (\ref{eq:eight}), using optimal control results  $V(x)$ should satisfy Hamilton Jacobi Bellman (HJB) equation \cite{opti,so:onetwo}
 \begin{equation}\label{eq:99}
min_{u_{1}\in{R^{m}}}(x^{T}F_{m}x+x^{T}Qx+u_{1}^{T}Ru_{1}+V_{x}^{T}(A(p_{0})x+Bu_{1})=0
\end{equation}
where $V_{x}=\frac{\partial{V}}{\partial{x}}$ and $u_{1}=K_{1}x$. The equation (\ref{eq:99}) reduces to 
   \begin{equation}\label{eq:11}
(x^{T}F_{m}x+x^{T}Qx+u_{1}^{T}Ru_{1}+V_{x}^{T}(A(p_{0})x+Bu_{1})=0
\end{equation}
\item[Step 2]
According to optimal control theory, the optimal input $u_{1}(t)$ should minimize the Hamiltonian \cite{opti} 
\begin{equation}\label{cd:12}
H(x(t),u_{1}(t),V_{x})=x^{T}Fx+x^{T}Qx+u_{1}^{T}Ru_{1}+V_{x}^{T}(A(p_{0})x+Bu_{1})
\end{equation} 
  which leads to
 \begin{equation}\label{eq:12}
 \frac{\partial H(x(t),u_{1}(t),V_{x},t)}{\partial u_{1}(t)}=2x^{T}K_{1}^{T}R+V_{x}^{T}B=0
\end{equation}
\item[Step 3]
 For solving an infinite-time linear quadratic regulator (LQR) problem, a quadratic function $V(x)=x^{T}Sx$ is defined,  where matrix $S>0$. With this choice  the HJB equation reduces to the following algebraic Riccati equation 
 \begin{equation}\label{ri}
SA(p_{0})+A(p_{0})^{T}S+F_{m}+Q-SBR^{-1}B^{T}S=0
\end{equation} 
The solution  $S$ of (\ref{ri}) is used to compute the optimal control input $u_{1}$  which is
  \begin{equation}\label{con12l}
  u_{1}(t)=-R^{-1}B^{T}Sx(t)=K_{1}x(t)
 \end{equation}
 \item[Step 4] Control gain $K_{1}$ for (\ref{mawe}) and aperiodic state information of original system  $x(t_{k})$ are used to compute the control law  
\begin{equation}\label{unc10}
u_{1}(t_{k})=K_{1}x({t_{k}})
\end{equation}
\end{description}
The above mentioned steps 1 to 4 are also adopted for the unmatched system (\ref{mismatched1}) with cost functional $J_{u}$ (\ref{mismatched2}). The control input $u_{2}(t)$ and auxiliary input $v(t)$ are computed if the following  are satisfied:  
\begin{eqnarray}\label{mis3}
x^{T}(F_{u}+\rho^{2}H+\beta^{2}I)x+u_{2}^{T}u_{2}+\rho^{2}v^{T}v+V_{x}^{T}(A(p_{0})x+Bu_{2}+\alpha(I-BB^{+})v)=0
\end{eqnarray}
\begin{equation}\label{miss34}
x^{T}K_{2}^{T}+V_{x}^{T}B=0
\end{equation}
\begin{equation}\label{miss35} 
2\rho^{2}x^{T}L^{T}+V_{x}^{T}\alpha(I-BB^{+})=0
\end{equation}
In LQR problem, the HJB equation (\ref{mis3}), reduces to a following  algebraic Riccati equation  
\begin{equation}\label{rim}
 \hat{S}A(p_{0})+A(p_{0})^{T}\hat{S}+F_{u}+\rho^{2}H+\beta^{2}I-\hat{S}(BB^{T}+\alpha^{2}\rho^{-2}(I-BB^{+})^{2})\hat{S}=0
 \end{equation}
 The solution $\hat{S}>0$ of (\ref{rim}) is used to compute control input $u_{2}$ and auxiliary input $v$ and given by  
  \begin{equation}\label{vf12}
 \begin{bmatrix}
u_{2}(t)\\
v(t)
\end{bmatrix}
=\begin{bmatrix}
-B^{T}\hat{S}\\
-\alpha\rho^{2}(I-BB^{+})\hat{S}
\end{bmatrix}
=\begin{bmatrix}
K_{2}\\
L
\end{bmatrix}x(t)
 \end{equation}
 The optimal controller gain $K_{2}$ for (\ref{mismatched1}) is used to generate the robust event-triggered control input for (\ref{mis12}) and it is written as
 \begin{equation}\label{unc09}
u_{2}(t_{k})=K_{2}x({t_{k}})
\end{equation}
  Now for event-triggered control it is important to design the event triggering instant such that uncertain system (\ref{nst1}) is ISS with a aperiodic control law (\ref{in12}). The approach for deriving the triggering law is discussed below.
 \subsubsection{Triggering condition design} Given an uncertain system (\ref{nst1}) with a linear controller (\ref{in12}) there must have an event-triggering instant $t_{k\in\mathbb{I}}$with  a positive inter-execution time $(t_{k+1}-t_{k}=\tau>0)$ such that the closed loop system (\ref{nst1}) is ISS. To prove this there must have an ISS Lyapunov function with the time derivative  in the form of (\ref{deft2}). The ISS condition in the form of (\ref{deft2}) helps to construct the event-triggering rule in-terms of measurement error norm $\|e(t)\|$ and  the state norm of the system  original system $\|x(t)\|$. To design the event-triggering law for matched system (\ref{eq:eight}) and unmatched system (\ref{mis12}), the ISS Lyapunov functions  are considered as follows: 
\begin{itemize}
\item ISS Lyapunov function for matched system:
\begin{equation}\label{eq:478}
V_{m}(x)=x^{T}Sx
\end{equation}
\item ISS Lyapunov function for unmatched system:
\begin{equation}\label{co12}
V_{u}(x)=x^{T}\tilde{S}x
\end{equation} 
\end{itemize} 
\section{Static event-triggered robust control}\label{se123}
 This section describes static event-triggering  law for both matched  and unmatched uncertain systems. Here the main results of this paper are stated in the form of following theorems and  proofs. 
 \begin{thrm}\label{th1}
Suppose the controller gain matrix $K_{1}$ is designed for the nominal system (\ref{mawe}) by minimizing the cost functional (\ref{eq:45}). The matched uncertain system (\ref{eq:eight}), with even-trigger based controller (\ref{unc10}), is ISS if there exist a static event occurring sequence $\{t_{k}\}_{k\in{\mathbb{I}}}$ given by
\begin{equation}\label{triwe}
t_{0}=0, \ t_{k+1}=inf\{{t\in{\mathbb{R}}|t>t_{k}\wedge\mu_{1}\|x\|-\|e\|}\leq0\}
\end{equation} 
where parameter $\mu_{1}$ is defined in (\ref{eq:1213}).
\end{thrm}
\begin{proof} \label{th01}  To prove ISS of (\ref{eq:eight}), it is necessary to simplify the derivative of ISS Lyapunov function  $V(x)$  in the form of (\ref{deft2}).  Here the expression  of  $V(x)$ is same as $V_{m}(x)$, as defined in (\ref{eq:478}). The time derivative of $V(x)$ along the trajectories of (\ref{eq:eight}) can be written as
 \begin{eqnarray}
\dot{V}(x)&=&V_{x}^{T}\dot{x}\label{VdotEqn} \\
&=&V_{x}^{T}(A(p_{0})x+BK_{1}x+B\phi(p)x+BK_{1}e) \nonumber
\end{eqnarray}
 Using  (\ref{eq:11}) and substituting $V_{x}^{T}=2x^{T}S$  
\begin{eqnarray*}
\dot{V}(x)&=&-x^{T}F_{m}x-x^{T}Qx-u^{T}Ru -2x^{T}K_{1}^{T}R\phi(p)x+2x^{T}SBK_{1}e \\
&=&- x^{T}((F_{m}-\phi(p)^{T}R\phi(p))+Q +
(K_{1}+\phi(p))^{T}R(K_{1}+\phi(p)))x+2x^{T}SBK_{1}e
\end{eqnarray*}

According to  Definition \ref{de1}, the condition (\ref{deft2}) holds if
\begin{eqnarray}\label{w23}
\alpha(\|x\|)=\frac{\lambda_{min}(Q_{1})}{2}\|x\|^{2}\\
\gamma(\|e\|)=\frac{2\|SBK_{1}K_{1}^{T}B^{T}S\|}{\lambda_{min}(Q_{1})}\|e\|^{2}\label{w83}
\end{eqnarray}
In the above expression, 
\begin{equation}\label{eq1000}
Q_{1}=(F_{m}-\phi(p)^{T}R\phi(p))+Q+(K_{1}+\phi(p))^{T}R(K_{1}+\phi(p))
\end{equation}
  From (\ref{deft2}),\ (\ref{w23}) and (\ref{w83}) it can be written that  the following triggering condition  need to be violated to update the control input.
\begin{eqnarray}\label{eq:1212}
\|e\|\leq\mu_{1}\|x\|
\end{eqnarray}
Here notation $\mu_{1}$ is used to represent
\begin{equation}\label{eq:1213}
\mu_{1}=\frac{\sigma\lambda_{min}(Q_{1})}{2\|SBK_{1}\|}
\end{equation}
Also (\ref{eq:1212}), (\ref{eq:1213}) suggest the time instant at which the event has occurred.
 \begin{equation}\label{eq:120}
 t_{0}=0,\ t_{k+1}=inf\{{t\in{\mathbb{R}}|t>t_{k}\wedge\mu_{1}\|x\|-\|e\|}\leq0\}
 \end{equation}
  Using (\ref{eq:120}), $\dot{V}(x)$ becomes
 \begin{equation*}
 \dot{V}(x)\leq(\sigma-1)\lambda_{min}(Q_{1})\|x\|^{2}
 \end{equation*}
 Therefore the uncertain static event-triggered system (\ref{eq:eight}) is stable  $\forall \ \sigma\in(0,1)$. 
\end{proof}
\begin{remark}
 It is seen from (\ref{un190}) that the first term of $Q_{1}$ is a positive definite matrix. The bound on final term of  $Q_{1}$ is derived as
$(K_{1}+\phi(p))^{T}R(K_{1}+\phi(p))\leqslant \|K_{1}^{T}RK_{1}\|+\lambda_{max}(RF_{m}).$ The  positiveness of all three terms ensure the positive definiteness of $Q_{1}$. 
\end{remark}
\begin{remark}
In the absence of uncertainty, $\phi(p)=0$, the expression  (\ref{triwe}) reduces to the similar results reported in \cite{something6,something}. In this sense, the proposed algorithm generalizes the existing results and also valid for uncertain systems.      
\end{remark}
The results for unmatched system is stated in the form of following theorem:
\begin{thrm}\label{th2}
Suppose the controller gain matrices $K_{2}$ and $L$ are designed for the nominal system (\ref{mismatched1}) by minimizing the cost functional (\ref{mismatched2}) and the inequality $\beta^{2}I-2\rho^2L^{T}L>0$ holds, then the uncertain system (\ref{mis12}), with event-based control law (\ref{unc09}) is ISS if there exist a static event occurring sequence $\{t_{k}\}_{k\in I}$ given by
\begin{equation}
t_{0}=0, t_{k+1}=inf\{{t\in{\mathbb{R}}|t>t_{k}\wedge\mu_{2}\|x\|-\|e\|}\leq0\}
\end{equation}
where the design parameter $\mu_{2}$ is defined in (\ref{miu1}).
\end{thrm}
\begin{proof}
 Assume $V_{u}(x)$ is an ISS Lyapunov function of (\ref{mis12}) and denote $V_{u}(x)$ by $V(x)$. The time derivative of $V(x)$ along the state-trajectory of (\ref{mis12}) is simplified as
\begin{eqnarray*}
\dot{V}(x)&=& V_{x}^{T}\dot{x}\\
&= & V_{x}^{T}(A(p_{0})x+BK_{2}x+BK_{2}e)+V_{x}^{T}\alpha(I-BB^{+})Lx
+ V_{x}^{T}\alpha(I-BB^{+})Le\\&&+V_{x}^{T}BB^{+}(A(p)-A(p_{0}))x+V_{x}^{T}(I-BB^{+})(A(p)-A(p_{0}))x+V_{x}^{T}\alpha(I-BB^{+})(Le+Lx)
\end{eqnarray*}
 Using (\ref{mis3}), (\ref{miss34}) and (\ref{miss35}) 
\begin{eqnarray*}
\dot{V}(x)&=&-x^{T}\{(F_{u}+\rho^{2}+\beta^{2}I)+K_{2}^{T}K_{2}+\rho^{2}L^{T}L+2K_{2}^{T}B^{T}(A(p)-A(p_{0}))\\&&+2\alpha^{-1}\rho^{2}L^{T}(A(p)-A(p_{0}))\}x+2x^{T}\tilde{S}BK_{2}e
\end{eqnarray*}
Implying the upper bound mentioned in (\ref{mis10}), (\ref{miswe}) and after simplification the above equality turns to the following inequality.
\begin{equation}\label{mis15}
\dot{V}(x)\leq-\lambda_{min}(Q_{2})\|x\|^{2}+\|\tilde{S}BK_{2}K_{2}^{T}B^{T}\tilde{S}\|\|e\|\|x\|
 \end{equation}
 Now as per \emph{Defination I} the ISS condition mentioned in (\ref{deft1}), (\ref{deft2}) holds if 

\begin{eqnarray}\label{al11}
\alpha (\|x\|)=\frac{\lambda_{min}(Q_{2})}{2}\|x\|^{2}\\
\gamma(\|e\|)=\frac{2\|\tilde{S}BK_{2}K_{2}^{T}B^{T}\tilde{S}\|}{\lambda_{min}(Q_{2})}\|e\|^{2}\label{all22}
\end{eqnarray}
 In the above expression (\ref{al11})
\begin{equation}\label{qq}
Q_{2}=\beta^{2}I-2\rho^2L^{T}L
\end{equation}
 By hypothesis of Theorem \ref{th2}, matrix $Q_{2}>0$.
Using (\ref{deft2}), (\ref{al11}) and (\ref{all22}) it can be concluded  that  control input should update if the following inequality is violated. 
\begin{equation}\label{mis16}
\|e\|\leq\mu_{2}\|x\|
\end{equation}
 Here parameter $\mu_{2}$ is defined as 
 \begin{equation}\label{miu1}
 \mu_{2}=\frac{\sigma\lambda_{min}(Q_{2})}{2\|\tilde{S}BK_{2}\|}
 \end{equation}
  Equation (\ref{mis16}), (\ref{miu1})  also suggest the  time instant when the condition (\ref{mis16}) dose not hold and it is expressed as
   \begin{equation}\label{conwe1}
t_{0}=0, t_{k+1}=inf\{{t\in{\mathbb{R}}|t>t_{k}\wedge\mu_{2}\|x\|-\|e\|}\leq0\}
\end{equation}
 Using (\ref{conwe1}) the trajectory of $V(x)$ will be bounded by 
 \begin{equation}\label{mis17}
 \dot{V}(x)\leqslant(\sigma-1)\lambda_{min}(Q_{2})\|x\|^{2}
 \end{equation}
 which ensure that $\dot{V}(x)$ is decreasing    $\forall\sigma\in(0,1)$. 
\end{proof}
Theorem \ref{th1} \& \ref{th2} ensure stability of uncertain systems (\ref{eq:eight}), (\ref{mis12}) by static event-triggering rule respectively. 
An algorithmic representation of static event-triggered control with matched system uncertainty  is given next.
 \begin{algorithm}[H]
    \caption{Static event-triggered control for matched uncertain system}\label{alg:calcqij}
    \begin{algorithmic}[1]
      \STATE Initialization: $t\Leftarrow0$, $x\Leftarrow x_{0}$ , $x(t_{k})\Leftarrow x_{0}$.
      \STATE Given: $A(p_{0})$, $B$, $F_{m}$, $\sigma$
      \STATE Compute $\|x\|$, $\|e\|$ and $\mu_{1}$ using (\ref{eq:1213})
      \IF{$\|e\|$ $\geq$ $\mu_{1}\|x\|$}
      \STATE  Transmit $x(t_{k})$ from system end to  controller end.
      \STATE Solve algebraic Riccati equation (\ref{ri}) and compute controller gain $K_{1}$
      \STATE Compute $u_{1}(t_{k})=K_{1}x(t_{k})$
      \STATE Update $u_{1}(t_{k})$ in  (\ref{unc10})
      \ELSE 
      \STATE $u_{1}(t)=u_{1}(t_{k-1})$
      \ENDIF
      
      \STATE Return to line 3
    \end{algorithmic}
  \end{algorithm}
\subsection*{Minimum time interval in between two consecutive events}
In event-triggered control inter execution time depends on the evolution of  $\|e\|/\|x\|$ with respect to time. At $t_{k}$ the ratio of $\|e\|/\|x\|$  is zero as measurement error $e=0$. The next event will occur at $t_{k+1}$, when the $\|e\|/\|x\|$ turns to $\mu_{1}$.  Using (\ref{conwe1}), the minimum time required to evolve $\|e\|/\|x\|$ from $0$ to $\mu_{1}$ defines the lower bound of the inter-event time $\{t_{k+1}-t_{k}\}_{\forall k\in\mathbb{I}}=\tau>0$. Here inter-event time $\tau$ should be always a non-zero positive time interval to avoid the so called Zeno behaviour\footnote{Infinite number of transmission and computation in a finite time \cite{zenohy}.}. The minimum time interval in between two consecutive events of proposed robust control mechanism is stated in the form of a theorem.
\begin{thrm}\label{th3}
$\forall \ \sigma\in(0,1)$, the system (\ref{eq:eight}) with triggering law (\ref{eq:120}) has strictly positive lower bound of inter-event time $\tau>0$ and it is expressed as
\begin{eqnarray}\label{tauer}
\tau=\frac{2}{\sqrt{L_{3}^{2}-4L_{2}L_{1}}}\ln\{\|\frac{2L_{2}\mu_{1}+L_{3}}{\sqrt{L_{3}^{2}-4L_{2}L_{1}}}\|
\|\frac{L_{3}+\sqrt{L_{3}^{2}-4L_{2}L_{1}}}{L_{3}-\sqrt{L_{3}^{2}+4L_{2}L_{1}}}\|\} 
\end{eqnarray}
 where $ L_{1}=\|(A(p_{0})+B\phi(p)+BK)\|$, $L_{2}=\|BK\|$ and $L_{3}=L_{1}+L_{2}.$
\end{thrm}

\begin{proof}
From \cite{something6} the time derivative of $\|e\|/\|x\|$ can be written as 
\begin{eqnarray}\label{er1}
\frac{d}{dt}\frac{\|e\|}{\|x\|}&\leq & \frac{\|e\|\|\dot{x}\|}{\|e\|\|x\|}+\frac{\|x\|\|\dot{x}\|}{\|x\|\|x\|}\frac{\|e\|}{\|x\|}\nonumber
\\
&&\leq(1+\frac{\|e\|}{\|x\|})\frac{\|\dot{x}\|}{\|x\|}
\end{eqnarray}
Applying triangular inequality of vector norm on (\ref{eq:eight}) 
\begin{equation}\label{er2}
\|\dot{x}(t)\|\leq\|(A(p_{0})+B\phi(p)+BK_{1})\|\|x\|+\|BK_{1}\|\|e\|
\end{equation}
Denoting $ L_{1}=\|(A(p_{0})+B\phi(p)+BK_{1})\|$ , $L_{2}=\|BK_{1}\|$ and the ratio $\frac{\|e\|}{\|x\|}=y$  the  inequality (\ref{er2}) is simplified as
\begin{equation}\label{nil34}
\frac{dy}{dt}\leq{L_{1}+(L_{1}+L_{2})y+L_{2}y^{2}}
\end{equation}
 Applying comparison lemma \cite{nst669} on (\ref{nil34}), the differential inequality (\ref{nil34}) turns to the following differential equality 
\begin{equation}\label{time4}
\dot{\Omega}=L_{1}+(L_{1}+L_{2})\Omega+L_{2}\Omega^{2}
\end{equation}
With  a initial value $\Omega(0,\Omega_{0})=\Omega_{0}$, the solution $\Omega(t,\Omega_{0})$ of ( \ref{time4}) must satisfy the inequality $y(t)\leq\Omega(t,\Omega_{0})$. Thus the inter-event time, $\tau$ is bounded by $\mathbb{R}^{+}$ time to evolve  $\Omega$  from $0$ to $\mu_{1}$. The expression of $\tau$ can be derived by solving (\ref{time4}).
\begin{eqnarray}\label{ta12}
\tau &=&\frac{2}{\sqrt{L_{3}^{2}-4L_{2}L_{1}}}ln\norm{ \frac{(N_{1}+N_{2})N_{3}}{(N_{1}+N_{3})N_{2}}}\nonumber ,  \forall \ \sqrt{L_{3}^{2}-4L_{2}L_{1}}>0
\end{eqnarray}
 From (\ref{ta12}) it is obvious that  $\tau$ has positive value as $N_{3}>N_{2}$.  
\end{proof}
\begin{remark}
 For unmatched uncertain system expression of $\tau$ is similar to (\ref{tauer}) but the value of $L_{1}$, \ $L_{2}$ and $L_{3}$ are   $ L_{1}=\|A(p_{0})+BB^{+}(A(p)-A(p_{0}))+(I-BB^{+})(A(p)-A(p_{0}))+BK\|$, $L_{2}=\|BK\|$ and $L_{3}=L_{1}+L_{2}.$ In both cases the expressions of $L_{1}$ and $L_{3}$ depend on unknown system uncertainty. Therefore to compute the lower bound of inter-event time, the value of $L_{1}$ and $L_{3}$ are computed in entire uncertainty region such that $\tau$ is minimal. It is possible as the bound of parametric uncertainty is known for both matched and mismatched systems.
 \end{remark}
 \section{Dynamic event-triggered robust control}\label{se14}
A. Girad  proposed a dynamic event-triggering mechanism  where a dynamic variable $\eta(t)\geq 0$ is added to achieve larger inter-event time \cite{something}.  The time evolution of new variable $\eta(t)$ is expressed by the following differential equation.
 \begin{equation}\label{eq:13}
\dot{\eta}(t)=-\beta(\eta{(t)})+\sigma\alpha(\|x(t)\|)-\gamma(\|e(t)\|)
\end{equation}
Here $\beta, \alpha,\gamma$  are smooth class $K_{\infty}$ functions and $\sigma\in(0,1)$.
  The preliminaries and efficiency of dynamic event-triggering mechanism  over the static one \cite{something6} is reported in \cite{something}. In this section the dynamic event-triggering approach is adopted to solve the present robust control problem with limited state and input information.  The dynamic event triggering instant generated for uncertain system is discussed through the following theorem.
 \begin{thrm}\label{theo2}
Suppose the controller gain matrix $K_{1}$ is designed for the nominal system (\ref{mawe}) by minimizing the cost functional (\ref{eq:45}), then the augmented matched system (\ref{eq:eight}), (\ref{eq:13}) with event-trigger based controller (\ref{unc10}), is asymptotically stable if there exist a dynamic event occurring sequence $\{t_{k}\}_{k\in{\mathbb{I}}}$  given by
\begin{eqnarray}
t_{0}=0, \ t_{k+1}&=&inf\{t\in{\mathbb{R}}|t>t_{k}\wedge\eta(t)\nonumber\\&&+\theta(\mu_{1}\|x\|-\|e\|)\leq0\}\label{eq:2011}
\end{eqnarray}
Where $\mu_{1}$ is defined in (\ref{eq:1213}). 
\end{thrm}
\begin{proof}\label{x}
From (\ref{eq:13}) the evolution of $\eta(t)$ with respect to time can be defined as
\begin{equation}\label{eq:1214}
 \dot{\eta}(t)=-\lambda\eta(t)+(\mu_{1}\|x\|-\|e\|)
 \end{equation}
 Now select $W(x(t),\eta(t))=V_{m}(x)+\eta(t)$ as a Lyapunov  function for augmented systems (\ref{eq:eight}), (\ref{eq:1214}). Then using (\ref{VdotEqn}) and (\ref{eq:1214})  the time derivative of $W(x)$ can be written as
\begin{eqnarray}\label{eq:1215}
\dot{W}(x)\leq(\sigma -1)\lambda_{min}(Q_{1})\|x\|^{2}-\lambda\eta(t)
\end{eqnarray}
 Form (\ref{eq:1215}), for any value of $\sigma\in(0,1)$ and $\eta(t)>0$ the closed loop system (\ref{eq:eight}) is ISS by dynamic event-triggering rule (\ref{eq:2011}). 
\end{proof}
\begin{remark}
Dynamic event-triggering law for unmatched system is not discussed here. In that case triggering law  will  depend on parameter $\mu_{2}$ and $Q_{2}$, defined in (\ref{miu1}), (\ref{qq}). The stability proof will be similar to the proof of Theorem \ref{theo2}. The mathematical expression of $\tau$  for dynamic, robust event-triggered strategies is not included  in this paper. The existence of a strictly positive inter-event time $\tau$ for dynamic event-triggered case is shown through  numerical results which is discussed in subsequent section.
\end{remark}
To realize dynamic event-triggered control following algorithm is considered. 
\begin{algorithm}
    \caption{Dynamic event-triggered control for matched uncertain system}\label{alg:calcqij}
    \begin{algorithmic}[1]
      \STATE Initialization: $t\Leftarrow0$, $x\Leftarrow x_{0}$ , $x(t_{k})\Leftarrow x_{0}$.
      \STATE Given value: $A(p_{0})$, $B$, $F_{m}$, $\sigma$,  $\eta$, $\theta$
      \STATE Compute $\|x\|$, $\|e\|$ and $\mu_{1}$ using  (\ref{eq:1213})
      \IF{ $\eta(t)+\theta(\mu_{1}\|x\| - \|e\|)\leq0$}
      \STATE  Transmit $x(t_{k})$ from system end to  controller end.
      \STATE Solve algebraic Riccati equation (\ref{eq:1213}) and compute controller gain $K_{1}$
      \STATE Compute $u_{1}(t_{k})=K_{1}x(t_{k})$
      \STATE Update $u_{1}(t_{k})$ in (\ref{mis12})
      \ELSE 
      \STATE $u_{1}(t)=u_{1}(t_{k-1})$
      \ENDIF
      
      \STATE Return to line 3
    \end{algorithmic}
  \end{algorithm} 
   \subsection{Guideline for a possible selection of design parameters}
The parameters $\theta$, $\sigma$ and $\lambda$ are used in (\ref{eq:2011})-(\ref{eq:1215}). These parameters mainly affect the  lower bound of inter-event time and convergence rate of system state. This subsection introduces a possible selection guideline of such parameters.  The convergence of closed loop system  (\ref{eq:eight}) and (\ref{mis12}) are directly associated with $\sigma$ as seen in (\ref{eq:1215}). As $\sigma\rightarrow 0 $ the convergence rate of (\ref{eq:eight}) [or (\ref{mis12})] equivalent  to the ideal closed loop system (\ref{eq:five}). The generated event number can also be controlled by varying the value of $\sigma$.  similarly the parameter $\theta$ has  contribution in the inter-event time $\tau$. A possible selection procedure of parameter $\theta$ is carried out by deriving a lower bound on $\tau$. The results are stated in form of a theorem.  
\begin{thrm}\label{thy}
$\forall \ \sigma\in(0,1)$, $\eta>0$ and $\theta>0$ the system (\ref{eq:eight}), (\ref{eq:1214}) with triggering law (\ref{eq:2011}) has strictly positive lower bound of inter-event time $\tau>0$ and it is expressed as
\begin{eqnarray}\label{tay}
\tau=\int_{0}^{\mu_{1}}\frac{d\Gamma}{\frac{L_{1}}{\mu_{1}}
+(L_{2}+\lambda)\Gamma+(\frac{1}{\theta}+L_{2}\mu_{1})\Gamma^{2}}
\end{eqnarray}
 where $ L_{1}=\|(A(p_{0})+B\phi(p)+BK)\|$,  $L_{2}=\|BK\|$ and $0<\theta\leq\frac{1}{L_{1}-\lambda}$. 
\end{thrm}
\begin{proof}
The proof of this theorem is inspired by \cite{something} and included in Appendix A.
\end{proof}
\begin{remark}
 The existence of positive inter-event time is guaranteed in the range of $0<\theta\leq\frac{1}{L_{1}-\lambda}$ and it helps to select the other parameter $\lambda$.
The value of $\lambda$ must satisfy $\lambda\leq L_{1}$ to make $\theta$ positive.  
\end{remark}
\begin{remark}
The expression of $\tau$ in (\ref{tay}) is derived for   $0<\theta\leq\frac{1}{L_{1}-\lambda}$. Similarly, an analytical bound on $\tau$ can also be derived for $\theta>\frac{1}{L_{1}-\lambda}$. Note that the value of scalar $L_{1}$ depends on uncertainty $\phi(p)$. Hence, it is difficult to say  the exact value of $\theta$ for which  event-triggering law (\ref{eq:2011}) have larger lower-bound $\tau$. But it is possible to compute $\tau$ as the uncertain region is known apriori.
\end{remark}
\begin{remark}
The analytical expression of $\tau$ for mismatched system is not addressed here. But it can be derived using similar approach with different $L_{1}, L_{2}$ and $L_{3}$.  The existence of  larger average inter-event time of dynamic event-triggering rule over the static one is shown numerically in the next section. 
\end{remark}
\section{Simulation Results and comparisons}\label{se15} This section explains two separate numerical examples to validate the theoretical results for both matched and unmatched event-triggered systems.   
\subsection{Example 1} A second order linear system with matched uncertainty is shown below:
\begin{eqnarray*}
\begin{bmatrix}
\dot{x}_{1}\\
\dot{x}_{2}
\end{bmatrix}
=
\begin{bmatrix}
0 & 1\\
1+p & p
\end{bmatrix}
\begin{bmatrix}
x_{1} \\
x_{2}
\end{bmatrix}
+
\begin{bmatrix}
0 \\
1
\end{bmatrix}
u_{1}
\end{eqnarray*}
Here  $A(p)=\begin{bmatrix}
0 & 1\\
1+p & p
\end{bmatrix}$,  $B=\begin{bmatrix}
0 \\
1
\end{bmatrix}$ and the uncertain vector $p\in[-2, 2]$. Event-triggered based closed loop system is given by 
\begin{eqnarray*}
\begin{bmatrix}
\dot{x}_{1}\\
\dot{x}_{2}
\end{bmatrix}
=
A(p)
\begin{bmatrix}
x_{1} \\
x_{2}
\end{bmatrix}
+
B
(K_{1}x+K_{1}e)
\end{eqnarray*}
To solve (\ref{ri}), matrices $Q$, $R$, $F_{m}$  are selected as 
$Q=\begin{bmatrix}
10 & 0 \\0 & 10
\end{bmatrix}$, $R=2$  and $
\phi(p)^{T}R\phi(p)=
\begin{bmatrix}
p\\
p
\end{bmatrix}R
\begin{bmatrix}
p & p
\end{bmatrix}
\leq
\begin{bmatrix}
8 & 8\\
8 & 8
\end{bmatrix}
=F_{m}
$
Using the above matrices, the solution of (\ref{ri}) is obtained as $S= \begin{bmatrix}
16.89 & 6.89\\
6.89 & 6.89
\end{bmatrix}$. The optimal controller gain $K_{1}$ is calculated as  $K_{1}=-R^{-1}B^{T}S=\begin{bmatrix}
-4.1623 & -4.1623
\end{bmatrix}$. The scalar $ \lambda_{min}(Q_{1})=10$ is calculated from (\ref{eq1000}) using maximum upper bound of the uncertain parameter $p=2$.
To compute $t_{k}$, the design parameters are selected as  $\sigma=0.98, \ \theta=0.1$ and $ k=0.6$.   The simulation is executed for 4.5 second with initial condition $[0.2, -0.35]^{T}$ for static and $[0.2, -0.35, 0.01]^{T}$ for dynamic event-triggered control. Here the parameter $p$ varies sinusoidally according to the equation $p=2sin(t)$.  Figures \ref{fi:111}, \ref{fi:112} show  the time evolution of control input and its update time instant for a conventional system which does not use the event-triggering mechanism. It can be seen in Figures \ref{fi:115} and \ref{fi:118} that the error norm is bounded by a state dependent threshold. This signifies that the closed loop system holds the ISS property.  Figure \ref{fi:114} and \ref{fi:117} show the total number of events and their corresponding positive inter-execution time for a matched system. From Table \ref{tab:title}, it is seen that the average inter-input computation time of event-based control is 50 times larger than the conventional continuous one.

\subsection{Example 2}
We consider a second order unmatched uncertain system (\ref{mis12})
where $\ A(p)=\begin{bmatrix}
0 & 1+p\\
1 & 0
\end{bmatrix}$, \ $A(p_{0})=\begin{bmatrix}
0 & 1\\
1 & 0
\end{bmatrix}$, and $B=\begin{bmatrix}
0 \\
1 
\end{bmatrix}$. Here the uncertain parameter is $p\in[-2,2]$ and it varies sinusoidally. 
 Using (\ref{mis1001}), (\ref{mis10}) the matched and unmatched components of uncertainty are calculated as 
$BB^{+}(A(p)-A(p_{0})=\begin{bmatrix}
0 & 0\\
0 & 0
\end{bmatrix}
$, $(I-BB^{+})(A(p)-A(p_{0}))=\begin{bmatrix}
p & p\\
0 & 0
\end{bmatrix}$, $F_{u}=\begin{bmatrix}
0 & 0\\ 0 & 0
\end{bmatrix}$ and $H=\begin{bmatrix}
4 & 4\\4 & 4
\end{bmatrix}$.
 The parameters of (\ref{mismatched2}) are selected as $\alpha=1,\rho=0.05$ and $\beta=10$. To solve (\ref{rim}), the algebraic Riccati equation is rewritten as
\begin{equation*}
\hat{S}\tilde{A}+\tilde{A}^{T}\hat{S}+\tilde{Q}-\hat{S}\tilde{B}\tilde{R}^{-1}\tilde{B}^{T}\hat{S}=0
\end{equation*}
where $\tilde{A}=A(p_{0})=\begin{bmatrix}
0 &1\\1 & 0
\end{bmatrix}, \ \tilde{B}=\begin{bmatrix}
B & \alpha(I-BB^{+})
\end{bmatrix}=\begin{bmatrix}
0 & 1 & 0\\1 & 0 & 0 
\end{bmatrix}, \ \tilde{Q}=F_{u}+\rho^{2}H+\beta^{2}I=\begin{bmatrix}
104 & 4\\4 &104
\end{bmatrix}$ and $\tilde{R}=\begin{bmatrix}
I & 0\\ 0 & \rho^{2}I
\end{bmatrix}$.

Using the positive definite solution of the Riccati equation $\hat{S}$, the feedback control input $u_{2}$ and $v$ are computed as
\begin{eqnarray*}
\begin{bmatrix}
u_{2}\\v
\end{bmatrix}
&=&\begin{bmatrix}
-9.9877 & -11.1232\\ -4.9215  & -0.4994\\0 &0
\end{bmatrix}\begin{bmatrix}
x_{1}\\x_{2}
\end{bmatrix}
\end{eqnarray*}
To compute the event-triggering conditions (\ref{mis16}) and (\ref{eq:2011}), parameters are selected as $\sigma=0.98,  \ k=0.6, \ \lambda=(1-\sigma)k$ and $\theta=0.1$. Here $\lambda_{min}(Q_{2})$  is calculated based on (\ref{qq}). The simulation is executed for 3.5 seconds with initial condition $[0.2, \ -0.35]^{T}$ and $[0.2, \ -0.35, \ 0.01 ]^{T}$ for static and dynamic cases respectively.

\begin{figure}
        \centering
        \begin{subfigure}{0.4\textwidth}
                \includegraphics[width=7 cm,height=4 cm]{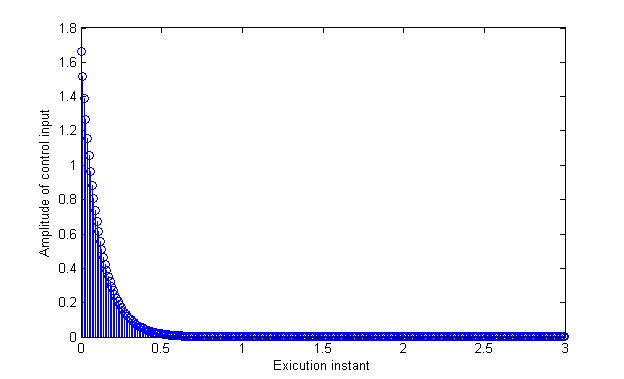}
\caption{Control input for without event-triggered control. }\label{fi:111}
        \end{subfigure}%
        ~ 
        \begin{subfigure}{0.4\textwidth}
                \includegraphics[width=7.5 cm,height=4 cm]{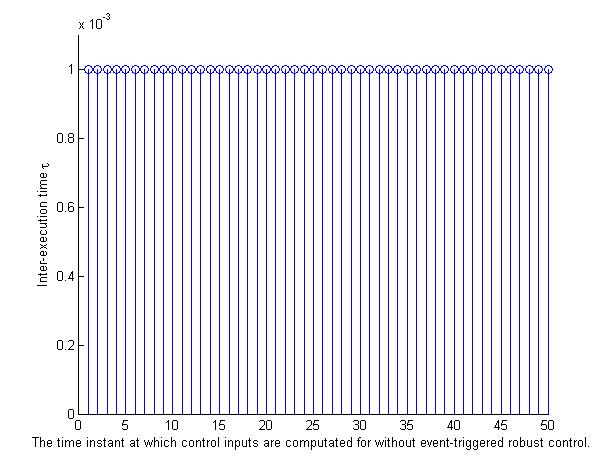}
\caption{ The time instant at which the control input are computated for without event-triggered robust control (for first few miliseconds). }\label{fi:112}
        \end{subfigure}
        ~ 
        \caption{Robust control of matched system with continious control law.}\label{fig:animals}
\end{figure}
\begin{figure}
        \centering
        ~ 
        \begin{subfigure}{0.4\textwidth}
                \includegraphics[width=7 cm,height=3.8 cm]{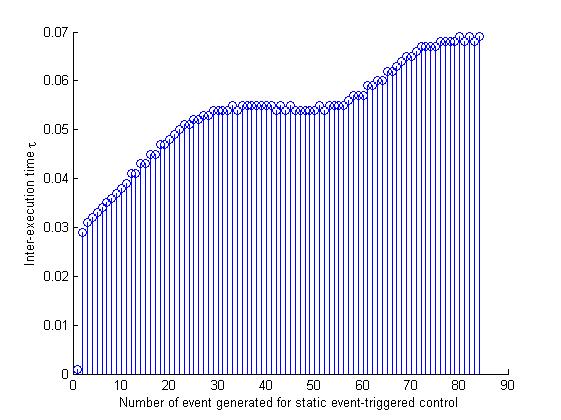}
\caption{Number of event occurrence and time interval between two consecutive events for static event-triggered control. }\label{fi:114}
        \end{subfigure}
        ~ 
          \begin{subfigure}{0.3\textwidth}
                \includegraphics[width=6 cm,height=3.8 cm]{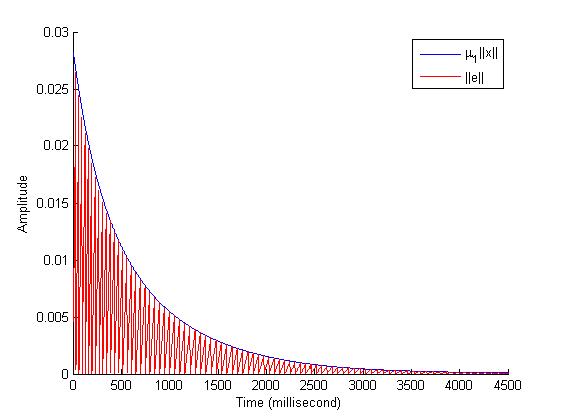}
\caption{Time evolution of $\|e\|$ (which is always within state state dependent threshold limit $\mu_{1}\|x\|$ ) of static event-triggered control.}\label{fi:115}
        \end{subfigure}
        
        \caption{Static event-triggered control with matched uncertainty.}\label{fig:animals}
\end{figure}
\begin{figure}
        \centering
        \begin{subfigure}{0.4\textwidth}
                \includegraphics[width=6 cm,height=3.8 cm]{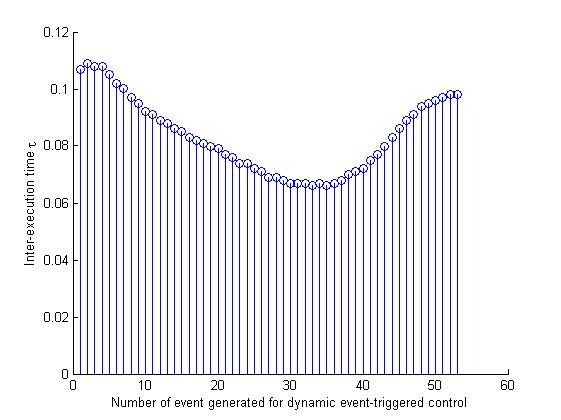}
\caption{Number of event occurrence and time interval between two consecutive events for dynamic event-triggered control. }\label{fi:117}
        \end{subfigure}
        ~ 
          \begin{subfigure}{0.3\textwidth}
                \includegraphics[width=6 cm,height=3.8 cm]{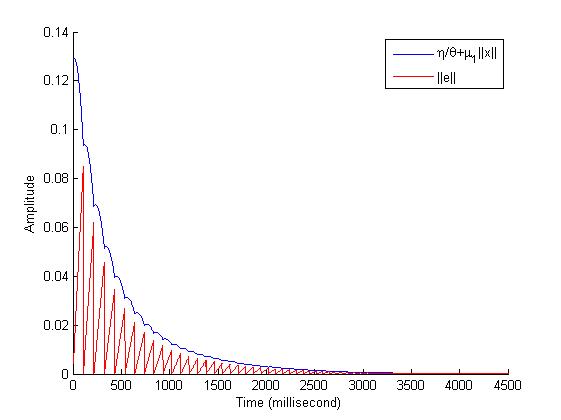}
\caption{Time evolution of $\|e\|$ (which is always within state state dependent threshold limit $\{\eta/\theta+\mu_{1}\|x\|\}$ ) of dynamic event-triggered control. }\label{fi:118}
        \end{subfigure}
\caption{Dynamic event-triggered control with matched uncertainty.}\label{fig:animals}
\end{figure}
\begin{figure}
        \centering
        \begin{subfigure}{0.4\textwidth}
                \includegraphics[width=7 cm,height=4 cm]{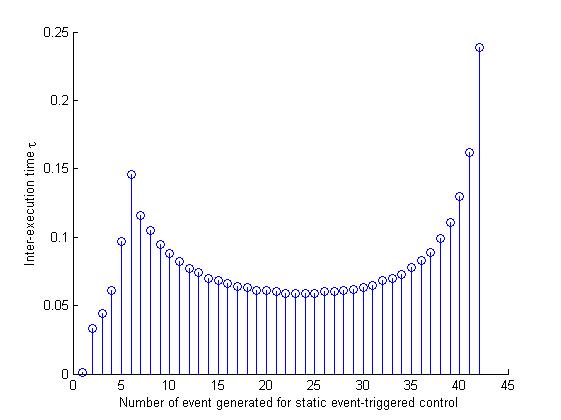}
\caption{Number of event occurrence and time interval between two consecutive events for static event-triggered control.}\label{fi:119}
        \end{subfigure}
        \begin{subfigure}{0.4\textwidth}
                \includegraphics[width=7 cm,height=4 cm]{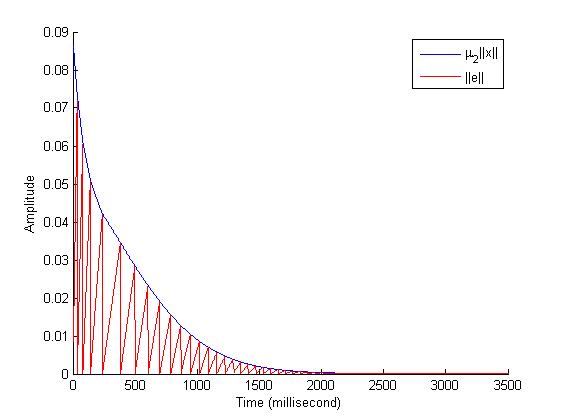}
\caption{Time evolution of $\|e\|$ (which is always within state state dependent threshold limit $\mu_{2}\|x\|$ ) for static event-triggered control. }\label{fi:1110}
        \end{subfigure}
         \caption{ Static event-triggered control with unmatched uncertainty.}\label{fig:animals}
\end{figure}
\begin{figure}
        \centering
        \begin{subfigure}{0.4\textwidth}
                \includegraphics[width=7 cm,height=4 cm]{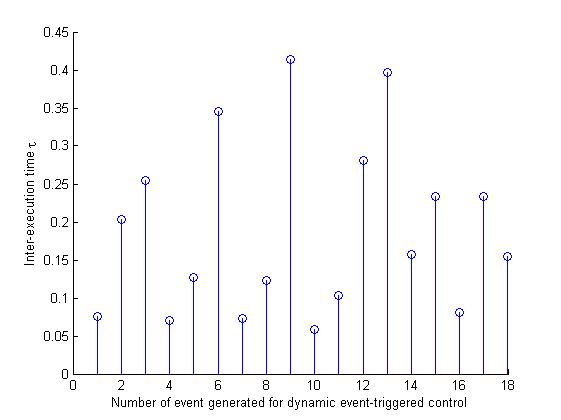}
\caption{Number of event occurrence and time interval between two consecutive events for dynamic event-triggered control. }\label{fi:1111}
        \end{subfigure}       
        \begin{subfigure}{0.4\textwidth}
                \includegraphics[width=7 cm,height=4 cm]{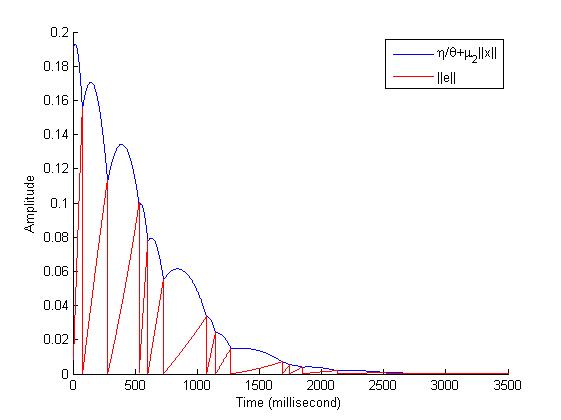}
\caption{Time evolution of $\|e\|$ (which is always within state state dependent threshold limit $\{\eta/\theta+\mu_{2}\|x\|\}$) for dynamic event-triggered control.}\label{fi:1112}
        \end{subfigure}   
        \caption{Dynamic event-triggered control with unmatched uncertainty.}\label{fig:animals}
\end{figure}

\begin{table}
\begin{center}
\caption {Comparative results of event-triggered and conventional continuous robust control approach} \label{tab:title} 
\begin{tabular}{ |l|l|l|l|l| }
\hline 
Control  mechanism & $\tau_{max}$ & $\tau_{min}$ &  $\tau_{avg}$ & $u_{total}$\\ \hline

Without event-triggered control & $0.001$ & $0.001$ &$0.001 $ & $3001$\\ \hline
 \multirow{2}{*}{Static  event-triggered control} & $ 0.06_{M}$ & $0.001_{M}$& $0.05_{M}$ &$86_{M}$ \\
 & $ 0.14_{U}$ & $0.01_{U}$ &$0.07_{U}$ &$43_{U}$ \\
  \hline
\multirow{2}{*}{Dynamic event-triggered control} & $0.10_{M}$ & $ 0.06_{M}$ & $0.08_{M}$ &$55_{M}$ \\
 & $0.44_{U}$& $0.05_{U}$ & $0.16_{U}$ &$19_{U}$\\
  \hline
\end{tabular}
\end{center}
\end{table}
Figure \ref{fi:1110} and  \ref{fi:1112} show that the error norm is bounded by a state dependent threshold for the unmatched system. The total number of events and their corresponding positive inter-execution time for unmatched system are shown in  Figure \ref{fi:119} and \ref{fi:1111}. For the purpose of comparison, Table \ref{tab:title} shows the inter-execution time for both event-triggered and without event-triggered control strategy. Here $\tau_{max}, \tau_{min},\tau_{avg}$ are the maximum, minimum and average inter-event time respectively. The scalar $u_{total}$ represents the total number of time instants at which control input is updated during the simulation period. According to Table \ref{tab:title}, the proposed event-triggered control strategy significantly reduces computation and transmission burden in the presence of parametric uncertainty. It also establishes the efficiency of dynamic event-triggering approach, as the average inter-event time is comparatively larger than the static one. The proposed event-based robust control law also ensures that the existence of positive inter-event time. 
\section{Conclusion}\label{se16}
This paper proposes a framework of event-triggered based robust control strategy for both matched unmatched uncertain system. The proposed control law is valid for a wider class of linear systems  in which event-triggering law is applicable. To design the event-triggering law, both static and dynamic event-triggering mechanisms  are adopted. The stability of closed loop event-triggering system is proved to be ISS for bounded variation of parameters. An analytical expression of static event-triggering mechanism  ensure that the proposed method is free from the Zeno behaviour. The lower bound of inter-event time for static event-triggered control is also derived. It is observed that the total number of control input computation and information transmission are very less in an event-based mechanism over  conventional system which do not use the event-triggered approach.\\
The detailed analysis of design parameters, like $\theta$, $\sigma$ and $k$ are not addressed in this paper. But theoretical analysis of these parameters are necessary to find out the optimal values and their effects on system performance. The numerical results show that average inter-execution time of dynamic event-triggered system is larger than the static one. However there need an exact mathematical expression of  inter-event time for dynamic event-triggering law. Self-triggered approach \cite{new1012}-\cite{new1111} may be considered to solve the robust control problem as a future work. 
\appendices
\section{}\label{ap121}
\begin{proof}[Proof of Lemma \ref{lemm1}]
The present approach translates the robust control problem stated in Section \ref{se12} to an equivalent optimal control problem \cite{so:one,so:1}. The proof consists of two parts. 
\subsection{Stability proof for matched uncertain system} 
If $V(x)$ is the Lyapunov function for (\ref{mawe}), then the time derivative of $V(x)$ along the state trajectory of (\ref{eq:five}) is
\begin{eqnarray*}
\dot{V}(x)&=&V_{x}\dot{x}
= V_{x}^{T}(A(p_0)x+BK_{1}x)+V_{x}^{T}B\phi(p)x
\end{eqnarray*}
Using (\ref{eq:11}), (\ref{eq:12}) and by adding and subtracting $x^{T}\phi(p)^{T}R\phi(p)x$, the $\dot{V}(x)$ reduces to
\begin{eqnarray}
\dot{V}(x)& = &-x^{T}((F_{m}-\phi(p)^{T}R\phi(p))+Q+(K_{1}+\phi(p))^{T}R(K_{1}+\phi(p))x\\
&&\leq -x^{T}x
\end{eqnarray}
Now $\dot{V}(x) \leq 0 $ for $x\neq 0$, which ensure asymptotic stability of the original uncertain matched system (\ref{eq:five}), (\ref{mis66}).
\subsection{Stability proof for unmatched uncertain system}
Similarly for a Lyapunov function $V(x)$, the $\frac{dV(x)}{dt}$ along the state trajectory of (\ref{eq:five}) is simplified using
 (\ref{mis10})-(\ref{miswe}) and  (\ref{mis3})-(\ref{miss35}) as
\begin{eqnarray}
\dot{V}(x)&\leqslant& -x^{T}(F_{u}+\rho^{2}H+\beta^{2}I)x - \rho^{2}x^{T}L^{T}Lx+x^{T}F_{u}x+\rho^{2}x^{T}L^{T}Lx+\rho^{2}x^{T}Hx+2\rho^{2}x^{T}L^{T}Lx\nonumber\\&&=-x^{T}(\beta^{2}I-2\rho^{2}L^{T}L)x
\end{eqnarray}

Now $\dot{V}(x)\leq 0$ for $x\neq0$ if the inequality $\beta^{2}I-2\rho^2L^{T}L>0$ holds. Therefore the closed loop system (\ref{eq:five}), (\ref{mis1001}) is asymptotically stable for all $p\in P$.

From the above two proofs, the optimal control input  $u_{1}$ of (\ref{mawe}) or $u_{2}$ of (\ref{mismatched1}) are the robust control input for (\ref{eq:five}) which proves
 the  Lemma\ref{lemm1}. 
 \end{proof}
\begin{proof}[Proof of Theorem \ref{thy}]
In dynamic event-triggered control the inter-event time depends on the  evolution of $\Gamma$, where  
\begin{eqnarray}
\Gamma=\frac{\theta\norm{e}}{\eta+\theta\mu_{1}\norm{x}}
\end{eqnarray}
The time derivative of $\Gamma$ along the direction of (\ref{eq:eight}), (\ref{eq:1214}) is simplified as 
\begin{eqnarray}\label{din}
\dot{\Gamma}&\leq & \frac{\theta (L_{1}\norm{x}+L_{2}\norm{e})}{(\eta+ \theta \mu_{1}\norm{x})}+\frac{\theta\norm{e}}{(\eta+\theta+\mu_{1}\norm{x})^{2}}\bigg\{\lambda \eta-\mu_{1}\norm{x}+\norm{e}+\theta\mu_{1}L_{1}\norm{x}+\theta\mu_{1}L_{2}\norm{e}\bigg\}\nonumber\\ &&
\leq\frac{ L_{1}}{\mu_{1}}
+(L_{2}+\lambda)\Gamma+(\frac{1}{\theta}+L_{2}\mu_{1})\Gamma^{2}+\frac{\theta\mu_{1}\norm x}{(\eta+\theta\mu_{1}\norm x)}\bigg(-\lambda-\frac{1}{\theta}+ L_{1}\bigg)\Gamma
\end{eqnarray}
Selecting $\theta=\frac{1}{L_{1}-\lambda}$, the final term in (\ref{din}) reduces to zero. 
Adopting  the similar steps as described in Section \ref{se123}, the lower bound of inter-event time for dynamic event-triggering is   
 \begin{eqnarray}\label{dse}
\tau=\int_{0}^{\mu_{1}}\frac{d\Gamma}{\frac{ L_{1}}{ \mu_{1}}
+(L_{2}+\lambda)\Gamma+(\frac{1}{\theta}+L_{2}\mu_{1})\Gamma^{2}}
\end{eqnarray}
To prove the positiveness of $\tau$ consider the function
\begin{eqnarray}\label{pfun}
g(\Gamma)= \frac{L_{1}}{ \mu_{1}}
+(L_{2}+\lambda)\Gamma+(\frac{1}{\theta}+L_{2}\mu_{1})\Gamma^{2}
\end{eqnarray} 
which has all positive coefficients. The function (\ref{pfun}) is a positive function  as $\dfrac{dg}{d\Gamma}>0, \ \ \forall \ \Gamma>0$. The  $\Gamma$ is a positive variable as $\eta, \ \theta,$ and $\mu_{1}$ all are positive. Therefore integration of (\ref{pfun}) over any positive interval is always positive. The expression (\ref{dse}) is also valid for $0<\theta<\frac{1}{L_{1}-\lambda}$. This ends the proof. 
\end{proof}
\section*{Acknowledgment}
The research work reported in this paper was partially supported by BRNS under the grant RP02348.

\end{document}